%% file: degree_seq.tex
\documentclass[bjps]{imsart}

\RequirePackage[OT1]{fontenc}
\usepackage{amsthm,amsmath,natbib}
\RequirePackage{hyperref}
\usepackage{natbib}

\arxiv{1601.02478}

\startlocaldefs
\numberwithin{equation}{section}
\theoremstyle{plain}

\endlocaldefs

\usepackage{graphicx}
\usepackage[utf8]{inputenc}
\usepackage{hyperref}

\input{notation.tex}

\begin{document}

\begin{frontmatter}
\title{Power-law decay of the degree-sequence probabilities of multiple random graphs
with application to graph isomorphism\thanksref{t1}}
\runtitle{Power-law decay of the degree-sequence probabilities of multiple random graphs}
\thankstext{t1}{This work has been partially funded
by CAPES, CNPq, and FAPERJ BBP grants.}

\begin{aug}
\author{\fnms{Jefferson Elbert} \snm{Sim\~oes}\thanksref{a}\ead[label=e1]{elbert@land.ufrj.br}},
\author{\fnms{Daniel R.} \snm{Figueiredo}\thanksref{a}\ead[label=e2]{daniel@land.ufrj.br}}
\and
\author{\fnms{Valmir C.} \snm{Barbosa}\thanksref{a}%
\ead[label=e3]{valmir@cos.ufrj.br}}

\runauthor{J. E. Sim\~oes et al.}

\affiliation[a]{Systems Engineering and Computer Science Program, COPPE, Federal University of Rio de Janeiro, Rio de Janeiro, Brazil}

\address{Systems Engineering and Computer Science Program, COPPE,\\
Federal University of Rio de Janeiro, Rio de Janeiro, Brazil,\\
\printead{e1,e2,e3}}

\end{aug}

\begin{abstract}
We consider events over the probability space generated by
the degree sequences of multiple independent Erdős-Rényi random graphs,
and consider an approximation probability space where
such degree sequences are deemed to be sequences of i.i.d.\@ random variables.
We show that, for any sequence of events with probabilities
asymptotically smaller than some power law in the approximation model,
the same upper bound also holds in the original model.
We accomplish this by extending an approximation framework
proposed in a seminal paper by McKay and Wormald.
Finally, as an example, we apply the developed framework
to bound the probability of isomorphism-related events
over multiple independent random graphs.
\end{abstract}




\end{frontmatter}

\section{Introduction}
The Erdős-Rényi random graph model,
also known as the $G(n,p)$ model \citep{ERErdosRenyi,ERGilbert}
is the most traditional probabilistic model for graphs.
In this model, a graph over $n$ vertices is randomly generated
by adding edges independently between each vertex pair with probability $p(n)$.
Despite its inability to model real-world networks,
its simplicity and the consequent analytical tractability
have allowed thorough theoretical analysis \citep{Bollobas-RandomGraphs}
and applications such as percolation models \citep{MeanFieldFrozenPerc}
and graph theory via the probabilistic method \citep{AlonSpencer-ProbMethod}.

One of the toughest challenges in understanding
the overall structure of the $G(n,p)$ random graph
is obtaining a precise characterization of its degree sequence.
The main reason for this is that,
even though the degrees of any two specific nodes
are only mildly correlated (due to the possible edge between them),
it is still a nontrivial task to compose these correlations
into a manageable joint distribution for the degrees.

Most results on this matter address the distribution of
the $t$-th largest degree, for some $t(n)$ generally bounded.
More recently, though, a framework has been set by~\cite{DegreeSeqRandomGraph}
for approximating the degree sequence
by a sequence of independent random variables,
with tight bounds on the error of the probabilities of events
estimated by this approximation.
This framework has been successfully applied in several contexts:
for instance, \cite{ChvatalsConditionCannotHold}
use it to analyze the middle degree asymptotics of random graphs,
which relates to Chvátal's condition for Hamiltonian graphs,
and \cite{DegSeqRandBipartiteGraphs} applies a similar technique
to analyze degrees in a random bipartite graph model.

In this paper, we consider the problems of
comparing the degree sequences of multiple random graphs,
and of approximating these degree sequences by
corresponding sequences of independent random variables.
Our main result (\autoref{full-approx-2})
directly relates power-law decaying probabilities in the two models:
any event sequence that has probability $o(n^{-a})$ in the approximation model
also has probability $o(n^{-a})$ in the original degree sequence model.
To achieve this, we extend the framework in~\cite{DegreeSeqRandomGraph}
to establish a relationship between the degree sequences of all graphs
and the corresponding independent sequences
through a series of intermediate approximations.
The stepwise error bounds, formally established by \autoref{approx-steps-2},
lay down a roadmap for handling asymptotic probabilities
of properties that compare the structures of a set of graphs.

As an example, we apply \autoref{full-approx-2}
to the problem of graph isomorphism.
Not only is this problem an interesting theoretical problem in its own right,
but it also has implications in practical problems such as
network privacy and anonymization \citep{PrivacyAnonNetworks}
and computer vision \citep{GraphMatchingCompVision}.
In particular, we show that, for a certain range of model parameters,
in a set of $k$ random graphs, there will not be an isomorphic pair
with probability $1-\binom{k}{2}o(n^{-1/2})$,
and they will not be all isomorphic with probability $1-o(n^{-k/2})$.

This paper is structured as follows:
in section~\ref{related} we review the degree sequence approximation framework,
detailing its steps and stating the main results used.
We then proceed to extending the framework to multiple independent random graphs,
providing corresponding statements and proofs in section~\ref{results}.
Our sample application will be presented in section~\ref{application},
where we apply the framework to the problem of isomorphism,
after which we conclude with some final remarks in section~\ref{remarks}.

In this paper, we use the following definitions for
the Bachmann-Landau family of asymptotic notations.
For any two real functions $f$, $g$:
\begin{itemize}
\item $f = o(g) \iff g = \omega(f) \iff \lim_{n\to\infty}\left|\frac{f(n)}{g(n)}\right| = 0$;
\item $f = O(g) \iff g = \Omega(f) \iff \limsup_{n\to\infty}\left|\frac{f(n)}{g(n)}\right| < \infty$;
\item $f = \Theta(g) \iff f = O(g) \land g = O(f)$.
\end{itemize}
\section{Related work} \label{related}

\cite{DegreeSeqRandomGraph} have previously formalized,
under quite loose constraints, the very intuitive result that
the degree sequence of a $G(n,p)$ random graph is similar to
a sequence of independent random variables,
each having distribution $\Bin(n-1,p)$.
This result takes the form of a number of theorems and lemmas,
each performing one of four steps in the approximation process
that is detailed in this section.
Notation will be kept as similar as possible to the original work.

For some fixed $n\in\Nat$, take the set $I_n = \{0,\ldots,n-1\}^{n}$
equipped with the discrete $\sigma$-algebra as our measurable space.
Let $d = (d_1, \ldots, d_n)$ be some element in this space.
Also, let $p = p(n)\in(0,1)$, and denote $N = \binom{n}{2}$ and $q = 1-p$.

In the \emph{binomial model} $\mathcal{B}_{n,p}$, $d$ is distributed
as a sequence of $n$ independent $\Bin(n-1,p)$ random variables.
This can be achieved by evaluating $d$
under the probability measure $\Prob_{\mathcal{B}_{n,p}} = \Bin(n-1,p)^{\ox n}$.
We would like to assert that this model is similar
to the degree sequence of a $G(n,p)$ random graph.
We call this the \emph{degree sequence model} ($\mathcal{D}_{n,p}$),
and denote by $\Prob_{\mathcal{D}_{n,p}}$ the probability measure
under which $d$ has this distribution.
Note that the sum of degrees in any graph is necessarily even,
which means $d$ will take, with probability 1, values on the set
$E_n = \{d\in I_n~:~M(d)\text{ is even}\}$
(where $M = M(d) = \|d\|_1$ is the sum of the components of $d$).

The approximation process requires three additional models
(with corresponding probability measures) that will perform
a transition from the binomial model to the degree sequence model,
with two of them making $d$ acquire properties from the degree sequence model
that are not present in the binomial model, 
and the third one acting as a technical middleman.
The first model is the \emph{even-sum binomial model} ($\mathcal{E}_{n,p}$).
It ensures that $d$ indeed takes values in $E_n$ with probability 1.
To ensure minimum distortion between probability of elements of $E_n$,
this model is simply set to be
the restriction of the binomial model to the set $E_n$.\footnote{
That is, the corresponding probability measure is
the measure for the binomial model conditional to the event $E_n$,
evaluated only on the events in $E_n$.}
Then, the \emph{weighted even-sum binomial model} ($\mathcal{E}'_{n,p}$)
ensures the stronger property that $M$ has the same distribution
as it does under the degree sequence model
(namely, that $M/2$ is distributed as $\Bin(N,p)$).
To insert as little interference as possible
into the relative probabilities of any two points in $E_n$,
the probabilities of all points $E_n$ are rescaled (or \emph{reweighted})
uniformly on each set $S_m = \{d\in E_n~:~M(d) = m\}$,
to make these sets have the desired probability.

To perform the bridge between $\mathcal{E}'_{n,p}$ and $\mathcal{E}_{n,p}$,
they have introduced the \emph{integrated model} $\mathcal{I}_{n,p}$,
which is essentially a ``noisy'' version of
the even-sum model $\mathcal{E}_{n,p}$.
The model $\mathcal{I}_{n,p}$ is obtained from $\mathcal{E}_{n,p}$
by switching from a fixed parameter $p$
to a random parameter $p'$ that quickly concentrates around $p$.
More specifically, $p'$ must be distributed as a truncated normal variable,
with expected value $p$, variance $pq/2N$,
and restricted to the unit interval.

We can informally summarize the approximation scheme as follows:
$$\Prob_{\mathcal{B}_{n,p}} \,\approx\, \Prob_{\mathcal{E}_{n,p}}
\,\approx\, \Prob_{\mathcal{I}_{n,p}} \,\approx\,
\Prob_{\mathcal{E}'_{n,p}} \,\approx\, \Prob_{\mathcal{D}_{n,p}}.$$

Now, for these approximations to work,
it is necessary for $p(n)$ to lie in a ``good behavior range'',
in which case $p = p(n)$ is said to be \emph{acceptable}.
The last approximation, in particular, is hard to tighten in general,
so the necessary conditions for this approximation to work
are brought into the definition of an acceptable function:

\begin{definition}
A function $p = p(n)$ is acceptable if the following conditions hold:
\begin{enumerate}
\item $pqN = \omega(n)\log n$;
\item there is a set $R_p(n) \subset E_n$ and a real function $\delta(n) = o(1)$
such that:
\begin{enumerate}
\item $\Prob_{\mathcal{D}_{n,p}}(R_p(n)), \Prob_{\mathcal{E}_{n,p}}(R_p(n))
= 1 - n^{-\omega(n)}$;
\item for every $d\in R_p(n)$, there is some $\delta_d$ such that
$|\delta_d| \leq \delta(n)$ and
$$\frac{\Prob_{\mathcal{D}_{n,p}}(d)}{\Prob_{\mathcal{E}'_{n,p}}(d)} =
\exp\left\{\frac{1}{4}\left(1-\frac{\gamma_2^2}{\lambda^2(1-\lambda)^2}\right)\right\}\cdot\exp\{\delta_d\},$$
where $\lambda(d) = M(d)/2N$ and $\gamma_2(d) = (n-1)^{-2} \sum_{i=1}^n(d_i-M(d))^2$.
\end{enumerate}
\end{enumerate}
\end{definition}

The second condition in this definition requires
a set $R_p(n)$ to exist in our sample space $E_n$,
with very large probability in $\mathcal{D}_{n,p}$ and $\mathcal{E}_{n,p}$
(the probability of its complement in both models
vanishes faster than any standard exponential),
in which the models $\mathcal{D}_{n,p}$ and $\mathcal{E}'_{n,p}$
uniformly agree to a ratio that approaches 1.
This condition is required for the proofs to be carried out,
though it has been conjectured by McKay and Wormald
that condition 1 in the definition is sufficient for $p(n)$ to be acceptable ---
to the best of our knowledge, this conjecture is still open.
For our purposes, they have identified an interesting regime for $p(n)$
in which these conditions hold:

\begin{theorem}
$p(n)$ is acceptable whenever $\omega(n)\log n/n^2 \leq pq \leq o(n^{-1/2})$.
\end{theorem}

The execution of this approximation scheme
has been broken down into a number of pieces with various levels of complexity,
so to fit different possibilities of applications.
In our particular case, we would like to ensure that
this scheme is well-suited for approximating probabilities
that vanish faster than power laws in $n$.
For this purpose, we extract the following results
from~\cite{DegreeSeqRandomGraph}, condensed in a single theorem.

\begin{theorem}
\label{approx-steps-1}
Let 
$\phi(x;\mu,\sigma^2)$ be the density function of the normal distribution,
and $V_{n,p} = \int_0^1 \phi(x;p,pq/2N)\dd x$.
Then the following statements hold:
\begin{enumerate}

\item For any event $A_n \subseteq E_n$,
$$\Prob_{\mathcal{E}_{n,p}}(A_n) =
\frac{2\Prob_{\mathcal{B}_{n,p}}(A_n)}{1+(q-p)^{2N}};$$
\label{bp-ep}

\item For any event $A_n \subseteq E_n$,
$$\Prob_{\mathcal{I}_{n,p}}(A_n) =
\frac{1}{V_{n,p}}\int_0^1 \phi(x;p,pq/2N)\Prob_{\mathcal{E}_{n,x}}(A_n)\dd x;$$
\label{ep-ip}

\item If $pqN\to\infty$ and $y = y(n) = o(\sqrt[6]{pqN})$, then
$$\Prob_{\mathcal{I}_{n,p}}(d) =
\Prob_{\mathcal{E}_{n,p}'}(d)
\left(1+O\left(\frac{1+|y|^3}{\sqrt{pqN}}\right)\right)$$
uniformly over $\{d\in E_n~:~|M(d)-2Np|\leq 2y\sqrt{Npq}\}$;
\label{ip-e'p}

\item If $\omega(n)\log n/n^2 \leq pq \leq o(n^{-1/2})$, then
there are sets $R_p(n),R'_p(n) \subseteq E_n$
and a real function $\delta(n) = o(1)$ such that:
\begin{enumerate}
\item $\Prob_{\mathcal{D}_{n,p}}(R_p(n)), \Prob_{\mathcal{D}_{n,p}}(R'_p(n))
= 1 - n^{-\omega(n)}$;
\item in $R'_p(n)$, $\gamma_2 = \lambda(1-\lambda)(1+o(1))$;
\item for every $d\in R_p(n)$, there is some $\delta_d$ such that
$|\delta_d| \leq \delta(n)$ and
$$\frac{\Prob_{\mathcal{D}_{n,p}}(d)}{\Prob_{\mathcal{E}'_{n,p}}(d)} =
\exp\left\{\frac{1}{4}\left(1-\frac{\gamma_2^2}{\lambda^2(1-\lambda)^2}\right)
\right\}\cdot\exp\{\delta_d\}.$$
\end{enumerate}
\label{e'p-dp}

\end{enumerate}
\end{theorem}

\begin{proof}
All results used in this proof
have been extracted from \cite{DegreeSeqRandomGraph},
to which we refer the reader for notation and statements.
Statement~\ref{bp-ep} is a particular case of corollary 4.3
taking $f=\Ind_{A_n}$ the indicator function of the event $A_n$,
simplified by theorem 4.2 and the observation that,
since $f = 0$ in $I_n\setminus E_n$, $f = \tilde{f}$.
Statement~\ref{ep-ip} is a rewriting of lemma 2.4,
consequence of the construction of $\Prob_{\mathcal{I}_{n,p}}$
from $\Prob_{\mathcal{E}_{n,p}}$
and an application of the law of total probability ---
we note that, for $x \in [0,1]$, $\phi(x;p,pq/2N)/V_{n,p}$
is the density function of the random parameter $p'$ used in the construction.
Statement~\ref{ip-e'p} simply restates theorem 3.6.
Statement~\ref{e'p-dp} comes from the definition of acceptability
and corollary 3.5, noting that the hypothesis implies $p(n)$ is acceptable.
\end{proof}

These properties of good approximation provided by \autoref{approx-steps-1}
suffice for our purposes, as they allow us to derive
the following relationship between the end models
$\mathcal{B}_{n,p}$ and $\mathcal{D}_{n,p}$.

\begin{theorem}
\label{full-approx-1}
Let $\{A_n\}_{n\in\Nat}$ be a sequence of events in $E_n$,
and assume $p$ satisfies $\omega(\log n/n) \leq p \leq o(n^{-1/2})$.
For any fixed $a > 0$,
$\Prob_{\mathcal{B}_{n,p}}(A_n) = o(n^{-a})$ implies
$\Prob_{\mathcal{D}_{n,p}}(A_n) = o(n^{-a})$.
\end{theorem}

Even though \autoref{full-approx-1} follows from the pieces of the approximation framework,
it was not proved at the occasion.
For brevity, we will not provide a proof for it, either,
though we note that each step in such proof is a simplified version
of the corresponding step in the proof of \autoref{full-approx-2},
which considers multiple random graphs, to be presented in the next section.

\section{Results} \label{results}
In several domains, we can identify problems that can be reduced to
understanding whether the structures of a set of given graphs are similar.
In this work, we consider the situation where
these graphs are instances of the $G(n,p)$ model,
with the same size but possibly with different values of $p$ ---
that is, a set of $k$ random graphs $G_1,\ldots,G_k$,
with $G_i$ distributed as $G(n,p_i)$ for some $p_i \in (0,1)$.
We also assume that these instances are independent.

Naturally, we would like to compare the degree sequences of these graphs,
as such comparison can be used as a proxy for more complicated properties.
Intuitively, it would be trivial that, since the multiple degree sequences are independent and
each of them can be individually approximated by i.i.d. sequences
with small errors on the corresponding probabilities of events,
the joint approximation of all degree sequences
should similarly yield a small error as well.
However, we find it essential that
this extension of the single-graph case be obtained formally.
As we see in what follows, even though such extension is indeed possible,
achieving it is far from trivial.

Before we proceed, let us introduce some notation.
For $\vec{p} = (p_1,\ldots,p_k)\in(0,1)^k$, denote by $\Prob_{\mathcal{B}_{n,\vec{p}}}$
the probability measure $\bigox_{i\in[k]}\Prob_{\mathcal{B}_{n,p_i}}$
over $I_n^k$ --- and similarly for measures in other models, over $E_n^k$.
Our goal is to perform the following approximation scheme:
$$\Prob_{\mathcal{B}_{n,\vec{p}}} \,\approx\, \Prob_{\mathcal{E}_{n,\vec{p}}}
\,\approx\, \Prob_{\mathcal{I}_{n,\vec{p}}} \,\approx\,
\Prob_{\mathcal{E}'_{n,\vec{p}}} \,\approx\, \Prob_{\mathcal{D}_{n,\vec{p}}}.$$

Let us stress that $\Prob_{\mathcal{D}_{n,\vec{p}}}$
is the joint distribution of the degree sequences of
mutually independent random graphs $G(n,p_1)$,\ldots,$G(n,p_k)$,
and $\Prob_{\mathcal{B}_{n,\vec{p}}}$ is the corresponding approximation
by $k$ independent sequences of i.i.d. random variables.

We will extend our notation further and write $\vec{q} = (q_1,\ldots,q_k)$ with $q_i = 1-p_i$,
and denote by $\vec{d} = (d_1,\ldots,d_k)$ some element of $I_n^k$.
Note that each coordinate $d_i$ of $\vec{d}$ is an integer sequence of length $n$.
We will also write $\lambda_i=\lambda_i(\vec{d})=M(d_i)/2N$
and $(\gamma_2)_i = (\gamma_2)_i(\vec{d}) = (n-1)^{-2}\sum_{j=1}^n((d_i)_j-M(d_i))^2$.

This allows us to state an extended version of \autoref{approx-steps-1} that holds for any $k \geq 1$:

\begin{theorem}
\label{approx-steps-2}
Let $\phi(x;\mu,\sigma^2)$ and $V_{n,p}$ be as in \autoref{approx-steps-1}.
Then the following statements hold:
\begin{enumerate}

\item For any event $A_n \subseteq E_n^k$,
$$\Prob_{\mathcal{E}_{n,\vec{p}}}(A_n) =
\frac{2^k\Prob_{\mathcal{B}_{n,\vec{p}}}(A_n)}
{\prod_{i\in[k]}[1+(q_i-p_i)^{2N}]};$$
\label{bp-ep2}

\item For any event $A_n \subseteq E_n^k$,
$$\Prob_{\mathcal{I}_{n,\vec{p}}}(A_n) =
\frac{1}{\prod_{i\in[k]}V_{n,p_i}}\int_{[0,1]^k}
\prod_{i\in[k]}\phi\left(x_i;p_i,\frac{p_iq_i}{2N}\right)
\Prob_{\mathcal{E}_{n,\vec{x}}}(A_n)\dd \vec{x},$$
where $\vec{x} = (x_1,\ldots,x_k)$.
\label{ep-ip2}

\item If $\min_{i\in[k]}\{p_iq_iN\}\to\infty$ and $y = y(n)$ is $o(\sqrt[6]{\max_{i\in[k]}\{p_iq_iN\}})$, then
$$\Prob_{\mathcal{I}_{n,\vec{p}}}(\vec{d}) =
\Prob_{\mathcal{E}'_{n,\vec{p}}}(\vec{d})
\left(1+\sum_{i\in[k]}O\left(\frac{1+|y|^3}{\sqrt{p_iq_iN}}\right)\right)$$
uniformly over $\{\vec{d}\in E_n^k~:~|M(d_i)-2Np_i|\leq 2y\sqrt{p_iq_iN}~\forall~i\in[k]\}$;
\label{ip-e'p2}

\item If $\omega(n)\log n/n^2 \leq p_iq_i \leq o(n^{-1/2})$ for each $i$, then
there are sets $S_{\vec{p}}(n)\subseteq E_n^k$ and $S'_{\vec{p}}(n) \subseteq E_n^k$
and a real function $\varepsilon(n) = o(1)$ such that:
\begin{enumerate}
\item $\Prob_{\mathcal{D}_{n,\vec{p}}}(S_{\vec{p}}(n)),
\Prob_{\mathcal{D}_{n,\vec{p}}}(S'_{\vec{p}}(n)) = 1 - n^{-\omega(n)}$;
\label{e'p-dp2-1}
\item in $S'_{\vec{p}}(n)$, $(\gamma_2)_i = \lambda_i(1-\lambda_i)(1+o(1))$ for each $i\in[k]$;
\label{e'p-dp2-2}
\item for every $\vec{d}\in S_{p,p'}(n)$,
there is some $\varepsilon_{\vec{d}}$ such that
$|\varepsilon_{\vec{d}}| \leq \varepsilon(n)$ and
\begin{align*}
\frac{\Prob_{\mathcal{D}_{n,\vec{p}}}(\vec{d})}
{\Prob_{\mathcal{E}'_{n,\vec{p}}}(\vec{d})} &=
\exp\left\{\frac{1}{4}\left(k-\sum_{i\in[k]}\frac{(\gamma_2)_i^2}{\lambda_i^2(1-\lambda_i)^2}
\right)\right\} 
\cdot\exp\{\varepsilon_{\vec{d}}\}.\end{align*}
\label{e'p-dp2-3}
\end{enumerate}
\label{e'p-dp2}

\end{enumerate}
\end{theorem}

\begin{proof}
See \autoref{proof-of-approx-steps}.
\end{proof}

Using the theorem's stepwise approximation through the models,
we can derive a general-purpose rule for vanishing probabilities of
events involving independent $G(n,p)$ random graphs,
similar to the one stated in \autoref{full-approx-1}.

\begin{theorem}
\label{full-approx-2}
Let $A_n$ be a sequence of events in $E_n^k$.
If $\vec{p}\in[0,1]^k$ satisfies $\min_{i\in[k]} p_i\geq \omega(\log n/n)$ and
$\max_{i\in[k]} p_i \leq o(n^{-1/2})$,
then $\Prob_{\mathcal{B}_{n,\vec{p}}}(A_n) = o(n^{-a})$ implies
$\Prob_{\mathcal{D}_{n,\vec{p}}}(A_n) = o(n^{-a})$
for any fixed $a > 0$.
\end{theorem}

\begin{proof}
Before anything, we note that our hypotheses imply that
$\min_{i\in[k]} p_i\geq \omega(1/n)$ and
$\max_{i\in[k]} p_i \leq o(1)$,
facts that we will use several times along the proof.
Let $a>0$ be fixed, and assume $\Prob_{\mathcal{B}_{n,\vec{p}}}(A_n) = o(n^{-a})$.

In agreement with the approximation scheme previously presented,
we will prove our assertion in four steps,
each addressing one of the following statements:

\begin{enumerate}
\item $\Prob_{\mathcal{B}_{n,\vec{p}}}(A_n) = o(n^{-a})$ implies
$\Prob_{\mathcal{E}_{n,\vec{p}}}(A_n) = o(n^{-a})$;
\item $\Prob_{\mathcal{E}_{n,\vec{p}}}(A_n) = o(n^{-a})$ implies
$\Prob_{\mathcal{I}_{n,\vec{p}}}(A_n) = o(n^{-a})$;
\item $\Prob_{\mathcal{I}_{n,\vec{p}}}(A_n) = o(n^{-a})$ implies
$\Prob_{\mathcal{E}'_{n,\vec{p}}}(A_n) = o(n^{-a})$;
\item $\Prob_{\mathcal{E}'_{n,\vec{p}}}(A_n) = o(n^{-a})$ implies
$\Prob_{\mathcal{D}_{n,\vec{p}}}(A_n) = o(n^{-a})$.
\end{enumerate}

\begin{description}
\item[Step 1]
Assume $\Prob_{\mathcal{B}_{n,\vec{p}}}(A_n) = o(n^{-a})$.
\autoref{approx-steps-2}(\ref{bp-ep2}) states that
$$\Prob_{\mathcal{E}_{n,\vec{p}}}(A_n) = \frac{2^k\Prob_{\mathcal{B}_{n,\vec{p}}}(A_n)}
{\prod_{i\in[k]}[1+(q_i-p_i)^{2N}]}.$$
For each $i\in[k]$, $p_i = \omega(1/n)$ implies that $2Np_i\to\infty$ and
$(q_i-p_i)^{2N} = (1-2Np_i/2N)^{2N} \to 0$.
There are finitely many such $i$, thus it holds that
$\Prob_{\mathcal{E}_{n,\vec{p}}}(A_n) \sim
2^k\Prob_{\mathcal{B}_{n,\vec{p}}}(A_n)$ and,
since $\Prob_{\mathcal{B}_{n,\vec{p}}}(A_n) = o(n^{-a})$,
it follows that $\Prob_{\mathcal{E}_{n,\vec{p}}}(A_n) = o(n^{-a})$.

\item[Step 2]
Assume $\Prob_{\mathcal{E}_{n,\vec{p}}}(A_n) = o(n^{-a})$.
We turn to the expression that links
$\mathcal{E}_{n,\vec{p}}$ to $\mathcal{I}_{n,\vec{p}}$,
presented in \autoref{approx-steps-2}(\ref{ep-ip2}).

The normalization constant $\prod_{i\in[k]} V_{n,p_i}$ is the probability that
$k$ independent $\Nor(p_1,p_1q_1/2N),\ldots,\Nor(p_k,p_kq_k/2N)$ random variables
assume values in $[0,1]$.
Standardizing these random variables and denoting by $Q(\cdot)$
the Q-function\footnote{The \emph{Q-function} is the tail distribution of a standard normal random variable.},
we have that, for any $i\in[k]$,
\begin{align*}
V_{n,p_i} &= Q\left(-\frac{p_i}{\sqrt{p_iq_i/2N}}\right)
- Q\left(\frac{q_i}{\sqrt{p_iq_i/2N}}\right) \\
&= Q\left(-\sqrt{\frac{2Np_i}{q_i}}\right) - Q\left(\sqrt{\frac{2Nq_i}{p_i}}\right) \\
&\to 1,
\end{align*}
where the limit comes from the facts that
$2Np_i/q_i=\omega(1)$ whenever $p_i = \omega(1/n)$ and
$2Nq_i/p_i=\omega(1)$ whenever $p_i = o(1)$.
Since there are finitely many $i$, it holds that $1/\prod_{i\in[k]}V_{n,p_i} = \Theta(1)$.

For the integral, we will split the domain of integration
into several rectangles and deal with them separately.
To simplify our notation, we denote our integrand by $g(\vec{x}) =
\prod_{i\in[k]}\phi(x_i;p_i,\frac{p_iq_i}{2N})
\Prob_{\mathcal{E}_{n,\vec{x}}}(A_n)$.

Pick some constant $c > a$, and let $\delta_i = \delta_i(n) = \sqrt{cq_i\log n/Np_i}$
for each $i\in[k]$.
Note that $np_i = \omega(\log n)$ implies
$\delta_i = \sqrt{cp_iq_i\log n/Np_i^2} = o(p_iq_i\log n/\log^2 n) = o(p_i)$.
Since $p_i = o(q_i)$ whenever $p_i = o(1)$,
it holds that $\delta_i < p_i,q_i$ for all $i\in[k]$ as long as $n$ is for large enough.
For such $n$, we can perform the following decomposition of $[0,1]^k$.

Split the $i$-th coordinate of $[0,1]^k$ into three intervals: a \emph{left} section $L_i = [0,p_i(1-\delta_i))$, a \emph{central} section $C_i = [p_i(1-\delta_i),p_i(1+\delta_i)]$ and a \emph{right} section $R_i = (p_i(1+\delta_i),1]$. Now, to each string $\Sigma\in\{L,C,R\}^k$, associate the rectangle obtained by taking the Cartesian product of corresponding intervals for each coordinate --- call this region $S_\Sigma$. This splitting procedure is illustrated in \autoref{subdomains} for $k = 2$.

\begin{figure}[h]
\begin{center}
\includegraphics[draft=false,width=0.8\textwidth]{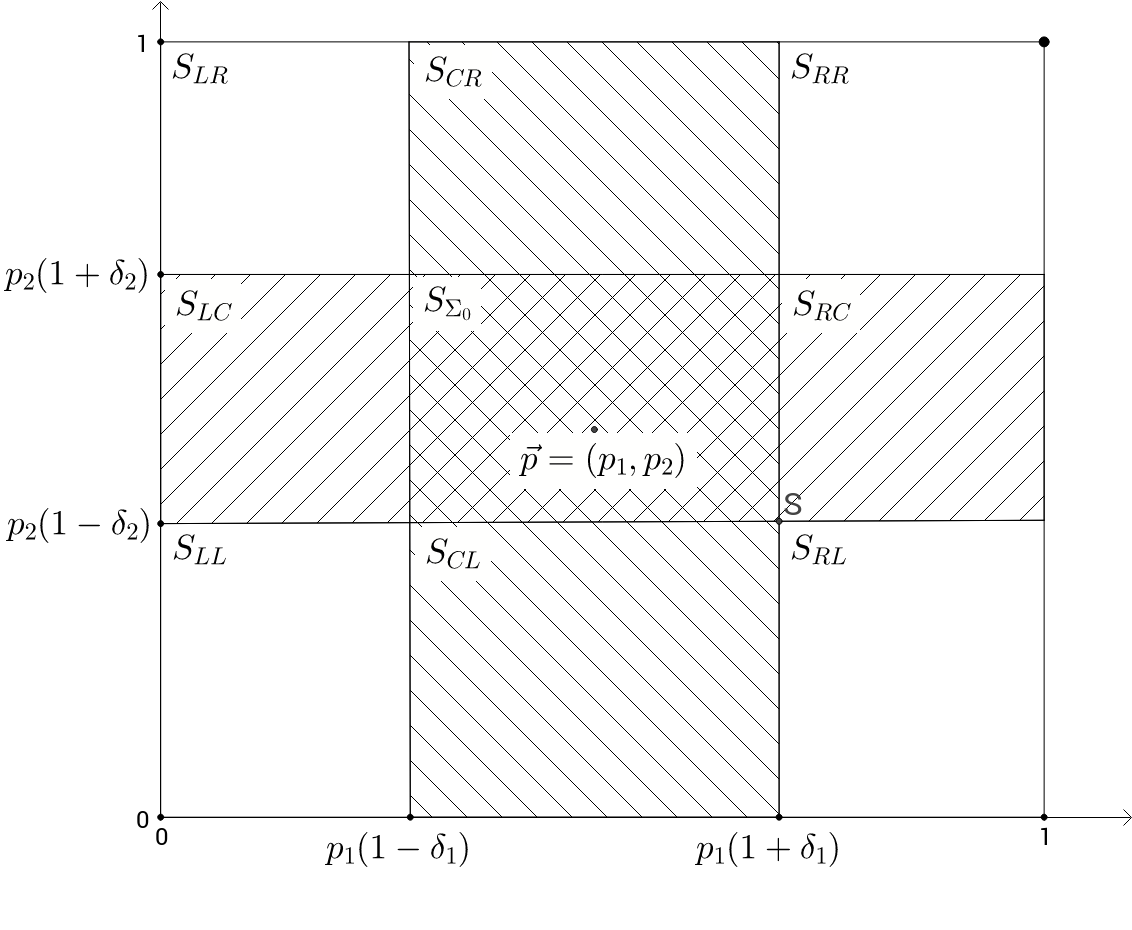}
\end{center}
\caption{Splitting $[0,1]^k$ into $3^k$ smaller domains of integration,
illustrated for the case $k = 2$.
Each region is assigned to a string in $\{L,C,R\}^k$ --- for instance, the string $CL$
corresponds to the lower central domain
$S_{CL} = C_1\times L_2 = [p_1(1-\delta_1),p_1(1+\delta_1)]\times[0,p_2(1-\delta_2))$.
The doubly hatched region $S_{\Sigma_0}$ corresponds to the intersection of all central sections
and corresponds to string $\Sigma_0 = CC{\ldots}C$.}
\label{subdomains}
\end{figure}

It is easy to see that $\biguplus_{\Sigma\in\{L,C,R\}^k} S_\Sigma = [0,1]^k$. This allows us to write $$\int_{[0,1]^k}\prod_{i\in[k]}\phi\left(x_i;p_i,\frac{p_iq_i}{2N}\right)
\Prob_{\mathcal{E}_{n,\vec{x}}}(A_n)\dd\vec{x} = \sum_{\Sigma\in\{L,C,R\}^k} g(\vec{x})\dd\vec{x}.$$

Denote $\Sigma = \Sigma^{(1)}\ldots\Sigma^{(k)}$, that is, $\Sigma^{(i)}$ denotes the $i$-th character of $\Sigma$.
There are now two cases to consider.
For the first case, assume $\Sigma\in\{L,C,R\}^k$ has at least one coordinate distinct from $C$,
namely $\Sigma^{(j)} \neq C$ for some $j\in[k]$.
Noting that $\Prob_{\mathcal{E}_{n,\vec{x}}}(A_n) \leq 1$, we can write:

\begin{align*}
\int_{\Sigma^{(1)}_1\times\cdots\times\Sigma^{(k)}_k} g(\vec{x})\dd\vec{x}
&= \int_{\Sigma^{(1)}_1\times\cdots\times\Sigma^{(k)}_k} 
\prod_{i\in[k]}\phi\left(x_i;p_i,\frac{p_iq_i}{2N}\right)\Prob_{\mathcal{E}_{n,\vec{x}}}(A_n)\dd\vec{x}\\
&\leq  \int_{\Sigma^{(1)}_1\times\cdots\times\Sigma^{(k)}_k} 
\prod_{i\in[k]}\phi\left(x_i;p_i,\frac{p_iq_i}{2N}\right)\dd\vec{x} \\
&= \prod_{i\in[k]}\left[\int_{\Sigma^{(i)}_i}\phi\left(x_i;p_i,\frac{p_iq_i}{2N}\right)\dd x_i\right].
\end{align*}

In the last expression, all terms are bounded by 1.
Moreover, since $\Sigma^{(j)}$ is either $L$ or $R$, it holds that
$$\int_{\Sigma^{(j)}_j}\phi\left(x_j;p_j,\frac{p_jq_j}{2N}\right)\dd x_j = o(n^{-a}),$$ as follows:

\begin{itemize}
\item if $\Sigma^{(j)} = L$, then
\begin{align*}
\int_{L_j}\phi\left(x_j;p_j,\frac{p_jq_j}{2N}\right)\dd x_j
&= 1-Q\left(-\frac{p_j\delta_j}{\sqrt{p_jq_j/2N}}\right) \\
&= Q\left(\frac{p_j\delta_j}{\sqrt{p_jq_j/2N}}\right) \\
& \leq \exp\left\{-\frac{Np_j\delta_j^2}{q_j}\right\} \\
& = \exp\{-c \log n\} = o(n^{-a}),
\end{align*}
where each step holds due to, respectively, definition, symmetry of tails and Chernoff bound for the Q-function, the choice of $\delta_j$ and the choice of $c$;
\item if $\Sigma^{(j)} = R$, then
\begin{align*}
\int_{R_ij}\phi\left(x_j;p_j,\frac{p_jq_j}{2N}\right)\dd x_j
&= Q\left(\frac{p_j\delta_j}{\sqrt{p_jq_j/2N}}\right) = o(n^{-a}),
\end{align*}
by a similar reasoning as in the previous case.
\end{itemize}

These two facts combined, imply $$\int_{\Sigma^{(1)}_1\times\cdots\times\Sigma^{(k)}_k} g(\vec{x})\dd\vec{x}
\leq o(n^{-a})\cdot 1\cdot1\cdot\cdots\cdot1 = o(n^{-a}).$$

For the second case, $\Sigma = CC{\ldots}C$, and a few prior comments are appropriate.
First, for any $i\in[k]$, note that, since $\delta_i = o(p_i)$,
for any $x_i = x_i(n)\in[p_i(1-\delta_i),p_i(1+\delta_i)]$, it is true that
$x_i = p_i(1+o(p_i))$ and, therefore, $x_i$ has the same asymptotics as $p_i$ ---
namely, $o(n^{-1/2}) \leq x_i \leq \omega(\log n/n)$.

Also, for any fixed $n$, $\Prob_{\mathcal{E}_{n,\vec{p}}}(A_n)$
is a continuous function of $p_i$ for each $i\in[k]$.
This comes from \autoref{approx-steps-2}(\ref{bp-ep2})
and the fact that $\Prob_{\mathcal{B}_{n,\vec{p}}}(A_n) =
\sum_{\vec{d}\in A_n} \Prob_{\mathcal{B}_{n,\vec{p}}}(\vec{d})$:
since the probability of each such $\vec{d}$ under measure $\Prob_{\mathcal{B}_{n,\vec{p}}}$
is a continuous function of $p_i$ for each $i$ (product of powers of $p_i$ and $1-p_i$
and constants with respect to $p_i$),
and the sum of these functions has a finite number of terms,
continuity of $\Prob_{\mathcal{B}_{n,\vec{p}}}(A_n)$
with respect to each $p_i$ follows; then, by \autoref{approx-steps-2}(\ref{bp-ep2}),
$\Prob_{\mathcal{E}_{n,\vec{p}}}(A_n)$ is the product between
$\Prob_{\mathcal{B}_{n,\vec{p}}}(A_n)$ and a continuous function of $p_i$,
so continuity of the former with respect to each $p_i$ also follows.

As a consequence of these results, for $\vec{x}\in C_1\times\cdots\times C_k$,
the function $\Prob_{\mathcal{E}_{n,\vec{x}}}(A_n)$,
being a continuous function over this compact set,
will attain a maximum value for some argument $\vec{y}(n) = (y_1,\ldots,y_k)(n)$ in this set.
Such $\vec{y}$ will, forcefully, satisfy
$\omega(\log n/n) \leq y_i \leq o(n^{-1/2})$,
which means that $\Prob_{\mathcal{E}_{n,\vec{y}}}(A_n) = o(n^{-a})$,
by our conclusion from the previous step.

That being said, we can assert that

\begin{align*}
\lefteqn{\int_{C_1\times\cdots\times C_k}g(\vec{x})\dd \vec{x}} \\
&\qquad= \int_{C_1\times\cdots\times C_k}\prod_{i\in[k]}\phi\left(x_i;p_i,\frac{p_iq_i}{2N}\right)
\Prob_{\mathcal{E}_{n,\vec{x}}}(A_n)\dd\vec{x} \\
&\qquad \leq \int_{C_1\times\cdots\times C_k}\prod_{i\in[k]}\phi\left(x_i;p_i,\frac{p_iq_i}{2N}\right)~\cdot \left[\max_{\vec{w}\in C_1\times\cdots\times C_k}
\Prob_{\mathcal{E}_{n,\vec{w}}}(A_n)\right]\dd\vec{x} \\
&\qquad = \int_{C_1\times\cdots\times C_k}\prod_{i\in[k]}\phi\left(x_i;p_i,\frac{p_iq_i}{2N}\right)
\Prob_{\mathcal{E}_{n,\vec{y}}}(A_n)\dd\vec{x} \\
&\qquad \leq \Prob_{\mathcal{E}_{n,\vec{y}}}(A_n)\cdot
\int_{[0,1]^k}\prod_{i\in[k]}\phi\left(x_i;p_i,\frac{p_iq_i}{2N}\right)\dd\vec{x} \\
&\qquad = o(n^{-a})\cdot 1 = o(n^{-a}).
\end{align*}

Thus, we conclude that
$$\Prob_{\mathcal{I}_{\vec{p}}}(A_n) =
\Theta(1)\cdot(3^k\cdot o(n^{-a})) = o(n^{-a}).$$

\item[Step 3]
Assume $\Prob_{\mathcal{I}_{n,\vec{p}}}(A_n) = o(n^{-a})$.
We begin by recalling that
$$\left(\frac{1}{2}M(S_1),\ldots,\frac{1}{2}M(S_k)\right)~\eqD~
\bigox_{i\in[k]}\Bin(N,p_i)~\text{ under }~\Prob_{\mathcal{E}'_{n,\vec{p}}}.$$

Define the event $N_n = \{|M(S_i)-2Np_i| < 2Np_i\cdot\varepsilon_i~\forall~i\in[k]\}$, with
$\varepsilon_i = (2Np_i)^{-5/12}$.
By the Chernoff bound, we have that, for all $i\in[k]$,
\begin{align*}
\Prob_{\mathcal{E}'_{n,\vec{p}}}(|M(S_i)-2Np_i|\geq 2Np_i\cdot\varepsilon_i)
& \leq 2e^{-2Np_i\varepsilon_i^2/6} = 2e^{-\frac{1}{6}(2Np_i)^{1/6}}.
\end{align*}
Thus, by the union bound, $\Prob_{\mathcal{E}'_{n,\vec{p}}}(\overline{N_n})
\leq 2\sum_{i\in[k]}e^{-\frac{1}{6}(2Np_i)^{1/6}}$.

Now, by definition, it holds in the event $N_n$ that, for each $i\in[k]$:
\begin{align*}
|M(S_i)-2Np_i| &< 2Np_i\cdot\varepsilon_i \\
& = (2Np_i)(2Np_i)^{-5/12}\frac{\sqrt{4Np_iq_i}}{\sqrt{4Np_iq_i}} \\
&= \left[\frac{(2Np_i)^{1/12}}{\sqrt{2q_i}}\right]\sqrt{4Np_iq_i}.
\end{align*}

These inequalities also hold in the event $A_n\cap N_n\subseteq N_n$.
Now, note that $(2Np_i)^{1/12}/\sqrt{2q_i} = o(\sqrt[6]{2Np_iq_i})$ for all $i\in[k]$,
which allows us to relate the probability of $A_n\cap N_n$ under measures
$\Prob_{\mathcal{E}'_{n,\vec{p}}}$ and $\Prob_{\mathcal{I}_{n,\vec{p}'}}$.
We choose
$y = \max_{i\in[k]}\{(2Np_i)^{1/12}/\sqrt{2q_i}\}$;
this choice of $y$ and $q_i = \Theta(1)$ imply that, for all $i\in[k]$,
$(1+|y|^3)/\sqrt{p_iq_iN} = o((Np_i)^{-1/2})+o(n^{3/8})/\omega(n^{1/2}) = o(1)$.
From these facts, using \autoref{approx-steps-2}(\ref{ip-e'p2}),
it follows that
\begin{align*}
\Prob_{\mathcal{E}'_{n,\vec{p}}}(A_n) &=
\Prob_{\mathcal{E}'_{n,\vec{p}}}(A_n\cap\overline{N_n}) +
\Prob_{\mathcal{E}'_{n,\vec{p}}}(A_n\cap N_n) \\
&\leq \Prob_{\mathcal{E}'_{n,\vec{p}}}(\overline{N_n}) +
\Prob_{\mathcal{E}'_{n,\vec{p}}}(A_n\cap N_n) \\
&\leq 2\sum_{i\in[k]}e^{-\frac{1}{6}(2Np_i)^{1/6}} \\
&\qquad+ \Prob_{\mathcal{I}_{n,\vec{p}}}(A_n\cap \overline{N_n})
\left(1+\sum_{i\in[k]}O\left(\frac{1+|y|^3}{\sqrt{p_iq_iN}}\right)\right)^{-1} \\
&\leq e^{-\omega(n)} + \Prob_{\mathcal{I}_{n,\vec{p}}}(A_n)(1+k\cdot o(1))^{-1} \\
&= o(n^{-a}) + o(n^{-a})(\Theta(1))^{-1} = o(n^{-a}). \\
\end{align*}

\item[Step 4]
Assume $\Prob_{\mathcal{E}'_{n,\vec{p}}}(A_n) = o(n^{-a})$.
Let the sets $S_{\vec{p}}(n), S'_{\vec{p}}(n)$ and the real function $\varepsilon(n)$
be as in \autoref{approx-steps-2}(\ref{e'p-dp2})
(note that our hypotheses about $\vec{p}$ imply
the hypotheses of this theorem are satisfied),
and define the set $T_{\vec{p}}(n) = S_{\vec{p}}(n)\cap S'_{\vec{p}}(n)$.
Then the following facts hold:

\begin{enumerate}
\item $\Prob_{\mathcal{D}_{n,\vec{p}}}(T_{\vec{p}}(n)) = 1 - n^{-\omega(n)}$, by the union bound;
\item for every $\vec{d}\in T_{\vec{p}}(n)$,
there is some $\varepsilon_{\vec{d}}$ such that
$|\varepsilon_{\vec{d}}| \leq \varepsilon(n)$ and
\begin{align*}
\frac{\Prob_{\mathcal{D}_{n,\vec{p}}}(\vec{d})}
{\Prob_{\mathcal{E}'_{n,\vec{p}}}(\vec{d})}
&= \exp\left\{\frac{1}{4}\left(k-\sum_{i\in[k]}\frac{(\gamma_2)_i^2}{\lambda_i^2(1-\lambda_i)^2}
\right)\right\} \cdot\exp\{\varepsilon_{\vec{d}}\};
\end{align*}
\item in $T_{\vec{p}}(n)$, $(\gamma_2)_i = \lambda_i(1-\lambda_i)(1+o(1))$ and
$\gamma'_2 = \lambda'(1-\lambda')(1+o(1))$;
\end{enumerate}

Using these facts, it follows that:
\begin{align*}
\Prob_{\mathcal{D}_{n,\vec{p}}}(A_n)
&= \Prob_{\mathcal{D}_{n,\vec{p}}}(A_n\cap\overline{T_{\vec{p}}(n)}) +
\Prob_{\mathcal{D}_{n,\vec{p}}}(A_n\cap T_{\vec{p}}(n)) \\
&\leq \Prob_{\mathcal{D}_{n,\vec{p}}}(\overline{T_{\vec{p}}(n)}) +
\sum_{\vec{p}\in A_n\cap T_{\vec{p}}(n)}
\Prob_{\mathcal{D}_{n,\vec{p}}}(\vec{p}) \\
&= n^{-\omega(n)} + \sum_{\vec{d}\in A_n\cap T_{\vec{p}}(n)}
\left[\Prob_{\mathcal{E}'_{n,\vec{p}}}(\vec{d})
\vphantom{\sum_{i\in[k]}\frac{(\gamma_2)_i^2}{\lambda_i^2(1-\lambda_i)^2}}~\cdot\right. \\
&\quad \left.\exp\left\{\frac{1}{4}
\left(k-\sum_{i\in[k]}\frac{(\gamma_2)_i^2}{\lambda_i^2(1-\lambda_i)^2}\right)\right\}
\cdot\exp\{\varepsilon_{\vec{d}}\}\right] \\
&= n^{-\omega(n)} + \left[\sum_{\vec{d}\in A_n\cap T_{\vec{p}}(n)}
\Prob_{\mathcal{E}'_{n,\vec{p}}}(\vec{d})\right]~\cdot \\
&\max_{\vec{d}\in A_n\cap T_{\vec{p}}(n)}
\exp\left\{\frac{1}{4}\left(k-\sum_{i\in[k]}\frac{(\gamma_2)_i^2}{\lambda_i^2(1-\lambda_i)^2}
\right)\right\}\cdot \\
&\quad \max_{\vec{d}\in A_n\cap T_{\vec{p}}(n)}
\exp\{\varepsilon_{\vec{d}}\} \\
&\leq n^{-\omega(n)}+\Prob_{\mathcal{E}'_{n,\vec{p}}}(A_n\cap T_{p,p'}(n))~\cdot\\
&\qquad\exp\left\{\frac{1}{4}(k-k(1+o(1))^2)\right\}
\cdot\exp\{\varepsilon(n)\}\\
&\leq o(n^{-a}) + \Prob_{\mathcal{E}'_{n,p,p'}}(A_n) \cdot
\exp\{o(1)\} \cdot \exp\{o(1)\} \\
&= o(n^{-a}) + o(n^{-a}) \cdot \Theta(1) \cdot \Theta(1) = o(n^{-a}).
\end{align*}

\end{description}
\end{proof}

\section{Example application} \label{application}

Our results so far establish an approximation scheme
between the degree sequences of $G(n,p)$ random graphs
and sequences of independent binomial random variables.
As such, it allows us to determine properties of random graphs
via a much simpler and more well-studied object.
Intuitively, if a graph property is related
to some feature of its degree sequence,
one can take this feature as a proxy for the original property,
analyze it assuming the degrees are independent
(that is, under the $B_{n,\vec{p}}$ model),
and use the framework to carry over the findings.

As an example application, consider the traditional problem of \emph{graph isomorphism}:
given two graphs $G_1$ and $G_2$, we would like to determine
whether or not they are isomorphic, that is,
whether there is an edge-preserving mapping between their vertex sets.
While this is an interesting problem,
and vastly explored in graph theory from a deterministic point of view,
it can also be studied in probabilistic settings,
such as that in which $G_1$ and $G_2$ are drawn from known random graph models.
In such settings, most of the work follows an algorithmic approach,
i.e., an algorithm is sought which correctly asserts a.a.s.
whether $G_1$ and $G_2$ are isomorphic.
The asymptotic correctness of the algorithm will, in general,
depend on the random graph model of choice, including its parameters.
Moreover, the use of \emph{canonical labeling algorithms} is often preferred~\citep[see][]{CanonLabelGraphs,RandGraphIso,CanonLabelLinearAvgTime,ProbAnalysisCanonNumberAlgGraphs,BeaconSetGraphIsomorphism}.

Here, by contrast, we follow a structural approach to the problem,
i.e., we would like to determine whether we can or cannot find, a.a.s., isomorphic graphs in a
sequence $G_1,G_2,\ldots,G_k$.
Problems of this nature require a mathematical solution
rather than an algorithmic solution\footnote{In particular,
in a regime where the input random graph instances are isomorphic a.a.s.,
the trivial algorithm that always outputs ``YES'' will be correct a.a.s.}.
In our example, we assume that all graphs at hand
are independent Erdős-Rényi random graphs.
In this case, the following result holds:
\begin{theorem}
\label{isomorphism}
Let $G_1,\ldots,G_k \eqD G(n,p_1)\ox\cdots\ox G(n,p_k)$ with
$\omega(\log n/n) \leq p_i \leq o(n^{-1/2})$ for all $i\in[k]$.
Then, $$\Prob[\text{at least two graphs are isomorphic}] \leq \binom{k}{2}\cdot o(n^{-1/2})$$
and $$\Prob[\text{all graphs are isomorphic}] \leq o(n^{-(k-1)/2}).$$
\end{theorem}

To prove this result, we will use an auxiliary graph-theoretic proposition.
Denote by $d_G(v)$ the degree of vertex $v$ in graph $G$.
For an arbitrary Borel set $B$ on the real line,
define $F_B(G) = |\{v\in V(G)~:~d_G(v)\in B\}|$, that is,
$F_B(G)$ counts the number of vertices in $G$ with degrees in $B$.
In general, for any finite sequence $d$ of length $|d|$,
denote $F_B(d) = \{i\in[|d|]~:~d[i]\in B\}$,
where $d[i]$ is the $i$-th component of $d$.
Note that, if $d$ is the degree sequence of graph $G$,
then $F_B(G) = F_B(d)$.

\begin{proposition}\label{iso-set}
If $G,G'$ are isomorphic, then for every Borel set $B$ on the real line,
$F_B(G) = F_B(G')$.
\end{proposition}

\begin{proof}
Let $f~:~V(G)\to V(G')$ be an isomorphism between $G$ and $G'$
(since $G$ and $G'$ are isomorphic, there is at least one such $f$).
$f$ is, by definition, bijective.
Also, since $f$ is edge-preserving, $f$ is also degree-preserving, that is,
$d_G(v) = d_{G'}(f(v))$ for any $v \in V(G)$.
Using these facts, for every Borel set $B$, we have
\begin{align*}
F_B(G) &= |\{v\in V(G)~:~d_G(v)\in B\}| \\
&= |\{v\in V(G)~:~d_{G'}(f(v))\in B\}| \\
&= |\{v'\in V(G')~:~d_{G'}(v')\in B\}| = F_B(G').
\end{align*}
\end{proof}

We can now proceed to the proof of \autoref{isomorphism}:

\begin{proof}[Proof of \autoref{isomorphism}]
Let $B_n = [\lfloor n\min_{i\in[k]}p_i\rfloor,\infty)$.
\autoref{iso-set} implies
\begin{align*}
& \Prob[\text{at least two graphs are isomorphic}] \\
&\qquad \leq \Prob[\exists~i\neq j~:~F_{B_n}(G_i)=F_{B_n}(G_j)] \\
&\qquad = \Prob[\exists~i\neq j~:~F_{B_n}(d_i)=F_{B_n}(d_j)] \\
&\qquad = \Prob_{\mathcal{D}_{n,\vec{p}}}[\exists~i\neq j~:~F_{B_n}(d_i)=F_{B_n}(d_j)],
\end{align*}
where $d_i$ is the degree sequence of graph $G_i$,
and the last equality holds by the distribution of $(d_1,\ldots,d_k)$ under $\Prob$.
We will show that the right-hand side of the inequality is $o(n^{-1/2})$,
and by virtue of \autoref{full-approx-2}, it is enough to show that
$\Prob_{\mathcal{B}_{n,\vec{p}}}[\exists~i\neq j~:~F_{B_n}(d_i)=F_{B_n}(d_j)] = o(n^{-1/2})$.

Now, fix an arbitrary $i\in[k]$. Note that, in the $\mathcal{B}_{n,\vec{p}}$ model,
all elements of the sequence $d_i$ belong to $B_n$ independently.
Furthermore, each such element is a $\Bin(n-1,p)$ random variable and belongs to $B_n$ with probability $\alpha_i = \alpha_i(n) > 1/2$
(since the median of $\Bin(n-1,p)$ is at most $\lceil(n-1)p\rceil$).
This implies that, under $\Prob_{\mathcal{D}_{n,\vec{p}}}$,
$F_{B_n}(d_i)\eqD\Bin(n,\alpha_n)$.

Moving on, let $b(k;n,p)$ be the mass function of a $\Bin(n,p)$ random variable.
Since $\alpha_i = \omega(1/n)$, it holds that
$\max_k b(k;n,\alpha_i) = o(n^{-1/2})$~\citep[see][]{ImprovRandGraphIsomorph}.
This implies
\begin{align*}
& \Prob_{\mathcal{B}_{n,\vec{p}}}[\exists~i\neq j~:~F_{B_n}(d_i)=F_{B_n}(d_j)] \\
&\qquad \leq\frac{1}{2}\cdot\sum_{i\in[k]} \Prob_{\mathcal{B}_{n,\vec{p}}}[\exists~j\neq i~:~F_{B_n}(d_i)=F_{B_n}(d_j)]\\
&\qquad = \sum_{x\leq n-1}\sum_{i\in[k]}\frac{1}{2}\cdot\Prob_{\mathcal{B}_{n,\vec{p}}}[F_{B_n}(d_i)=x,\exists~j\neq i:F_{B_n}(d_j)=x] \\
&\qquad = \sum_{i\in[k]}\sum_{x\leq n-1}\frac{1}{2}\cdot b(x;n,\alpha_i)\cdot \Prob_{\mathcal{B}_{n,\vec{p}}}[\exists~j\neq i:F_{B_n}(d_j)=x] \\
&\qquad \leq \sum_{i\in[k]}\left[\sum_{x\leq n-1}\frac{1}{2}\cdot b(x;n,\alpha_i)\right]\cdot \max_{x' \leq n-1}\Prob_{\mathcal{B}_{n,\vec{p}}}[\exists~j\neq i:F_{B_n}(d_j)=x'] \\
&\qquad \leq \sum_{i\in[k]}\frac{1}{2}\cdot (k-1)\cdot o(n^{-1/2})
=\binom{k}{2}o(n^{-1/2}). \\
\end{align*}

This proves the first inequality. For the second one, we note that
\begin{align*}
& \Prob[\text{all graphs are isomorphic}] \leq \Prob_{\mathcal{D}_{n,\vec{p}}}[\forall~i\neq j~:~F_{B_n}(d_i)=F_{B_n}(d_j)],
\end{align*}
and by \autoref{full-approx-2} it suffices to show that the probability of this event
under the $\mathcal{B}_{n,\vec{p}}$ model is $o(n^{-(k-1)/2})$.
This statement, in turn, holds since

\begin{align*}
& \Prob_{\mathcal{B}_{n,\vec{p}}}[\forall~i\neq j~:~F_{B_n}(d_i)=F_{B_n}(d_j)] \\
&\qquad \leq\sum_{x\leq n-1} \Prob_{\mathcal{B}_{n,\vec{p}}}[\forall~j\neq i~:~F_{B_n}(d_i)=x]\\
&\qquad \leq\sum_{x\leq n-1}\Prob_{\mathcal{B}_{n,\vec{p}}}[F_{B_n}(d_1)=x]\prod_{i\in[k]\setminus\{1\}}\Prob_{\mathcal{B}_{n,\vec{p}}}[F_{B_n}(d_i)=x] \\
&\qquad \leq\left[\sum_{x\leq n-1}\Prob_{\mathcal{B}_{n,\vec{p}}}[F_{B_n}(d_1)=x]\right]\max_{x'\leq n-1}\prod_{i\in[k]\setminus\{1\}}\Prob_{\mathcal{B}_{n,\vec{p}}}[F_{B_n}(d_i)=x'] \\
&\qquad \leq 1\cdot\prod_{i\in[k]\setminus\{1\}}\max_{x'\leq n-1}\Prob_{\mathcal{B}_{n,\vec{p}}}[F_{B_n}(d_i)=x'] = (o(n^{-1/2}))^k = o(n^{-(k-1)/2}).
\end{align*}
\end{proof}


\section{Final remarks} \label{remarks}

In this paper, we have considered the degree sequences of
$k$ independent Erdős-Rényi random graphs and
an approximation model in which such degrees are considered to be independent.
We have formally shown that any sequence of events in the approximation model
with probability smaller than a power law
will have this upper bound carried over to the original degree-sequence model.
It would be worthy of further analysis to determine whether this also holds
when $k$ is not a constant function of $n$.
We conjecture that it does as long as $k$ grows slowly enough, possibly any $k(n) = o(\log n)$.

\appendix
\section{Proof of \autoref{approx-steps-2}} \label{proof-of-approx-steps}

\begin{description}
\item[Statement~\ref{bp-ep2}]
Let $\mathcal{F}$ be the family of subsets of $E_n^k$
for which the statement's equality holds.
We will prove that (i) $\mathcal{F}$ contains all rectangles
(i.e., events of the form $R = R_1\times\cdots\times R_k$, with each $R_i\subset E_n$)
and (ii) $\mathcal{F}$ is a $\lambda$-system.
This is enough since, by Dynkin's theorem,
$\mathcal{F}$ must contain the $\sigma$-algebra generated by the rectangles,
which is the discrete $\sigma$-algebra over $E_n^k$.

For the first claim, for any rectangle $R = R_1\times\cdots\times R_k$ in $E_n^k$,
by \autoref{approx-steps-1}(\ref{bp-ep}), we have that
\begin{align*}
\Prob_{\mathcal{E}_{n,\vec{p}}}(R) &=
\Prob_{\mathcal{E}_{n,p_1}}\ox\cdots\ox\Prob_{\mathcal{E}_{n,p_k}}(R_1\times\cdots\times R_k) \\
&= \prod_{i\in[k]}\Prob_{\mathcal{E}_{n,p_i}}(R_i) \\
&= \prod_{i\in[k]}\frac{2\Prob_{\mathcal{B}_{n,p_i}}(R_i)}{1+(q_i-p_i)^{2N}} \\
&= \frac{2^k\prod_{i\in[k]}\Prob_{\mathcal{B}_{n,p_i}}(R_i)}{\prod_{i\in[k]}[1+(q_i-p_i)^{2N}]} \\
&= \frac{2^k\Prob_{\mathcal{B}_{n,\vec{p}}}(R)}{\prod_{i\in[k]}[1+(q_i-p_i)^{2N}]}.
\end{align*}
Therefore, $\mathcal{F}$ contains all rectangles.

For the second claim, note that $\mathcal{F}$ contains $E_n^k$,
since it is a rectangle;
$\mathcal{F}$ is closed by complements,
since for any $A\in\mathcal{F}$, it holds that
\begin{align*}
\Prob_{\mathcal{E}_{n,\vec{p}}}(E_n^k\setminus A) 
&= \Prob_{\mathcal{E}_{n,\vec{p}}}(E_n^k) - \Prob_{\mathcal{E}_{n,\vec{p}}}(A) \\
&= \frac{2^k\Prob_{\mathcal{B}_{n,\vec{p}}}(E_n^k)}
{\prod_{i\in[k]}[1+(q_i-p_i)^{2N}]}
- \frac{2^k\Prob_{\mathcal{B}_{n,\vec{p}}}(A)}
{\prod_{i\in[k]}[1+(q_i-p_i)^{2N}]} \\
&= \frac{2^k\Prob_{\mathcal{B}_{n,\vec{p}}}(E_n^k\setminus A)}
{\prod_{i\in[k]}[1+(q_i-p_i)^{2N}]}, \\
\end{align*}
and $E_n^k\setminus A\in\mathcal{F}$;
and $\mathcal{F}$ is also closed by disjoint enumerable unions,
since for any sequence $B_1,B_2,\ldots$ in $\mathcal{F}$,
if $B_1,B_2,\ldots$ are disjoint, then
\begin{align*}
\Prob_{\mathcal{E}_{n,\vec{p}}}\left(\biguplus_{j=1}^\infty B_j\right)
&=\sum_{j=1}^\infty\Prob_{\mathcal{E}_{n,\vec{p}}}(B_j)\\
&= \sum_{j=1}^\infty\frac{2^k\Prob_{\mathcal{B}_{n,\vec{p}}}(B_j)}
{\prod_{i\in[k]}[1+(q_i-p_i)^{2N}]} \\
&= \frac{2^k\sum_{j=1}^\infty\Prob_{\mathcal{B}_{n,\vec{p}}}(B_j)}
{\prod_{i\in[k]}[1+(q_i-p_i)^{2N}]} \\
&= \frac{2^k\Prob_{\mathcal{B}_{n,\vec{p}}}(\uplus_{j=1}^\infty B_j)}
{\prod_{i\in[k]}[1+(q_i-p_i)^{2N}]}, \\
\end{align*}
and $\uplus_{i=1}^\infty B_i\in\mathcal{F}$.
Since $\mathcal{F}$ satisfies the three requirements,
by definition, $\mathcal{F}$ is a $\lambda$-system.

\newpage
\item[Statement~\ref{ep-ip2}]
We follow the same strategy as in statement~\ref{bp-ep2}.
Let $\mathcal{G}$ be the family of subsets of $E_n^k$
for which the statement's equality is true.
First, take a rectangle $R = R_1\times\cdots\times R_k$ in $E_n^k$.
Using \autoref{approx-steps-1}(\ref{ep-ip}) yields
\begin{align*}
\Prob_{\mathcal{I}_{n,\vec{p}}}(R)
&= \Prob_{\mathcal{I}_{n,p_1}}\ox\cdots\ox\Prob_{\mathcal{I}_{n,p_k}}(R_1\times\cdots\times R_k) \\
&= \prod_{i\in[k]}\Prob_{\mathcal{I}_{n,p_i}}(R_i) \\
&= \prod_{i\in[k]}\frac{1}{V_{n,p_i}}\int_0^1
\phi\left(x_i;p_i,\frac{p_iq_i}{2N}\right)\Prob_{\mathcal{E}_{n,x_i}}(R_i)\dd x_i\\
&= \left[\prod_{i\in[k]}\frac{1}{V_{n,p_i}}\right]\int_{[0,1]^k}
\left[\prod_{i\in[k]}\phi\left(x_i;p_i,\frac{p_iq_i}{2N}\right)\prod_{i\in[k]}\Prob_{\mathcal{E}_{n,x_i}}(R_i)\right]\dd \vec{x}\\\\
&= \frac{1}{\prod_{i\in[k]}V_{n,p_i}}\int_{[0,1]^k}
\prod_{i\in[k]}\phi\left(x_i;p_i,\frac{p_iq_i}{2N}\right)
\Prob_{\mathcal{E}_{n,\vec{x}}}(R)~\dd \vec{x},
\end{align*}
which means $\mathcal{G}$ contains all rectangles, since $R$ was arbitrary.

Secondly, $\mathcal{G}$ satisfies the three requirements
of the definition of $\lambda$-systems:
it contains $E_n^k$, since it is a rectangle;
it is closed under complements, since for any $A\in\mathcal{G}$,
\begin{align*}
&\Prob_{\mathcal{I}_{n,\vec{p}}}(E_n^k\setminus A) \\
&\qquad= \Prob_{\mathcal{I}_{n,\vec{p}}}(E_n^k) - \Prob_{\mathcal{I}_{n,\vec{p}}}(A) \\
&\qquad= \frac{1}{\prod_{i\in[k]}V_{n,p_i}}
\int_{[0,1]^k}\prod_{i\in[k]}\phi\left(x_i;p_i,\frac{p_iq_i}{2N}\right)
\Prob_{\mathcal{E}_{n,\vec{x}}}(E_n^k)\dd\vec{x} \\
&\qquad\quad-\frac{1}{\prod_{i\in[k]}V_{n,p_i}}
\int_{[0,1]^k}\prod_{i\in[k]}\phi\left(x_i;p_i,\frac{p_iq_i}{2N}\right)
\Prob_{\mathcal{E}_{n,\vec{x}}}(A)\dd\vec{x} \\
&\qquad= \frac{1}{\prod_{i\in[k]}V_{n,p_i}}
\int_{[0,1]^k}\prod_{i\in[k]}\phi\left(x_i;p_i,\frac{p_iq_i}{2N}\right)
[\Prob_{\mathcal{E}_{n,\vec{x}}}(E_n^k)-\Prob_{\mathcal{E}_{n,\vec{x}}}(A)]\dd\vec{x} \\
&\qquad= \frac{1}{\prod_{i\in[k]}V_{n,p_i}}
\int_{[0,1]^k}\prod_{i\in[k]}\phi\left(x_i;p_i,\frac{p_iq_i}{2N}\right)
\Prob_{\mathcal{E}_{n,\vec{x}}}(E_n^k\setminus A)\dd\vec{x} \\
\end{align*}
and $E_n^k\setminus A\in\mathcal{G}$;
and $\mathcal{G}$ is also closed by disjoint enumerable unions,
since for any sequence $B_1,B_2,\ldots$ in $\mathcal{G}$,
if $B_1,B_2,\ldots$ are disjoint, then
\begin{align*}
&\Prob_{\mathcal{I}_{n,\vec{p}}}\left(\biguplus_{j=1}^\infty B_j\right) \\
&\qquad= \sum_{j=1}^\infty\Prob_{\mathcal{I}_{n,\vec{p}}}(B_j) \\
&\qquad= \sum_{j=1}^\infty\frac{1}{\prod_{i\in[k]}V_{n,p_i}}
\int_{[0,1]^k}\prod_{i\in[k]}\phi\left(x_i;p_i,\frac{p_iq_i}{2N}\right)
\Prob_{\mathcal{E}_{n,\vec{x}}}(B_j)\dd\vec{x} \\
&\qquad= \frac{1}{\prod_{i\in[k]}V_{n,p_i}}
\int_{[0,1]^k}\prod_{i\in[k]}\phi\left(x_i;p_i,\frac{p_iq_i}{2N}\right)
\left[\sum_{j=1}^\infty\Prob_{\mathcal{E}_{n,\vec{x}}}(B_j)\right]\dd\vec{x} \\
&\qquad= \frac{1}{\prod_{i\in[k]}V_{n,p_i}}
\int_{[0,1]^k}\prod_{i\in[k]}\phi\left(x_i;p_i,\frac{p_iq_i}{2N}\right)
\Prob_{\mathcal{E}_{n,\vec{x}}}(\uplus_{j=1}^\infty B_j)\dd\vec{x} \\
\end{align*}
and $\uplus_{i=1}^\infty B_i\in\mathcal{G}$.
Since $\mathcal{G}$ is a $\lambda$-system
and contains all rectangles, by Dynkin's theorem,
it must also contain the $\sigma$-algebra generated by the rectangles,
which is the discrete $\sigma$-algebra over $E_n$.

\item[Statement~\ref{ip-e'p2}]
Take $\vec{d} = (d_1,\ldots,d_k)\in E_n^k$ satisfying
$|M(d_i)-2Np_i|\leq 2y\sqrt{p_iq_iN}$ for each $i\in[k]$.
Using \autoref{approx-steps-1}(\ref{ip-e'p}) we can write
\begin{align*}
\Prob_{\mathcal{I}_{n,\vec{p}}}(\vec{d})
&= \prod_{i=1}^k\Prob_{\mathcal{I}_{n,p_i}}(d_i) \\
&= \prod_{i=1}^k \Prob_{\mathcal{E}_{n,p_i}'}(d_i)
\left(1+O\left(\frac{1+|y|^3}{\sqrt{p_iq_iN}}\right)\right).
\end{align*}

Note that the inequality from \autoref{approx-steps-1}(\ref{ip-e'p})
was applied $k$ times, once for each $\mathcal{I}_{n,p_i}$.
Since each inequality is uniform in its respective domain ---
$\{d_i\in E_n~:~|M(d_i)-2Np_i|\leq 2y\sqrt{p_iq_iN}\}$ ---,
the resulting inequality is uniform in the set
$\{\vec{d}\in E_n^k~:~|M(d_i)-2Np_i|\leq 2y\sqrt{p_iq_iN}~\forall i\}$.
Algebraic manipulations yield
\begin{align*}
\Prob_{\mathcal{I}_{n,\vec{p}}}(\vec{d})
&= \prod_{i\in[k]}\Prob_{\mathcal{E}_{n,p_i}'}(d_i)
\left(1+O\left(\frac{1+|y|^3}{\sqrt{p_iq_iN}}\right)\right) \\
&= \Prob_{\mathcal{E}_{n,\vec{p}}'}(\vec{d})\prod_{i\in[k]}
\left(1+O\left(\frac{1+|y|^3}{\sqrt{p_iq_iN}}\right)\right).
\end{align*}

Now, for each $i\in[k]$, $y = o(\sqrt[6]{p_iq_iN})$ implies
$(1+|y|^3)/\sqrt{p_iq_iN} = o(1)$.
Therefore, in expanding the product in the last expression,
the first-order terms dominate all higher-order terms.
This yields:
$$\Prob_{\mathcal{I}_{n,\vec{p}}}(\vec{d}) = \Prob_{\mathcal{E}_{n,\vec{p}}'}(\vec{d})
\left(1+\sum_{i\in[k]} O\left(\frac{1+|y|^3}{\sqrt{p_iq_iN}}\right)\right),$$
which is the desired result.

\item[Statement~\ref{e'p-dp2}]
This proof will follow by construction.
Under the stated assumptions for $p_i$,
there exist sets $R_{p_i}(n), R'_{p_i}(n)\subseteq E_n$ and a real function $\delta_i(n)$
satisfying the conditions of \autoref{approx-steps-1}(\ref{e'p-dp}),
with $(\delta_i)_{d_i}$ for each $d_i\in R_{p_i}(n)$ in condition (b).
Note that the functions $\delta_i(n)$ are positive real functions and
will not necessarily be equal for equal arguments.

Now, take:
\begin{align*}
S_{\vec{p}}(n) &= R_{p_1}(n)\times\cdots\times R_{p_k}(n), \\
S'_{\vec{p}}(n) &= R'_{p_1}(n)\times\cdots\times R'_{p_k}(n), \\
\varepsilon(n) &= \sum_{i\in[k]}\delta_i(n).
\end{align*}

We will show the desired results hold
for $S_{p,p'}$, $S'_{p,p'}$ and $\varepsilon$,
using properties of $R_{p_i}, R'_{p_i}$ and $\delta_i$
thoroughly in the next steps:

\begin{enumerate}
\item[\ref{e'p-dp2-1}.]
Note that
\begin{align*}
\Prob_{\mathcal{D}_{n,\vec{p}}}(S_{\vec{p}}(n)) &=
\prod_{i\in[k]}\Prob_{\mathcal{D}_{n,p_i}}(R_{p_i}(n)) \\
&= \prod_{i\in[k]}(1-n^{-\omega(n)}) \\
&= 1-n^{-\omega(n)},
\end{align*}
and similarly for $\Prob_{\mathcal{D}_{n,\vec{p}}}(S'_{\vec{p}}(n))$.

\item[\ref{e'p-dp2-2}.] Note that, for any $\vec{d} = (d_1,\ldots,d_k)\in S'_{\vec{p}}$,
it holds that $d_i\in R'_{p_i}(n)$ for each $i\in[k]$, each of which imply
$(\gamma_2)_i = \lambda_i(1-\lambda_i)(1+o(1))$.

\item[\ref{e'p-dp2-3}.]
By construction, for any $\vec{d} = (d_1,\cdots,d_k)\in S_{\vec{p}}(n)$,
it holds that $d_i\in R_{p_i}(n)$ for each $i\in[k]$.
Therefore, taking $\varepsilon_{\vec{d}} = \sum_{i\in[k]}(\delta_i)_{d_i}$,
it holds that $|\varepsilon_{\vec{d}}| \leq \sum_{i\in[k]}|(\delta_i)_{d_i}|
\leq \sum_{i\in[k]}\delta_i(n)=\varepsilon(n)$, and
\begin{align*}
\frac{\Prob_{\mathcal{D}_{n,\vec{p}}}(\vec{d})}
{\Prob_{\mathcal{E}'_{n,\vec{p}}}(\vec{d})}
&=\frac{\prod_{i\in[k]}\Prob_{\mathcal{D}_{n,p_i}}(d_i)}
{\prod_{i\in[k]}\Prob_{\mathcal{E}'_{n,p_i}}(d_i)} \\
&=\prod_{i\in[k]}\exp\left\{\frac{1}{4}\left(1-\frac{(\gamma_2)_i^2}{\lambda_i^2(1-\lambda_i)^2}\right)
\right\}\cdot\exp\{(\delta_i)_{d_i}\} \\
&=\exp\left\{\frac{1}{4}\left(k-\sum_{i\in[k]}\frac{(\gamma_2)_i^2}{\lambda_i^2(1-\lambda_i)^2}
\right)\right\}
\cdot\exp\{\varepsilon_{\vec{d}}\}
\end{align*}

\end{enumerate}
\end{description}
\qed

\bibliographystyle{imsart-nameyear}
\bibliography{degree_seq}

\end{document}

%% file: notation.tex
\usepackage{amsmath,amssymb}
\usepackage{thmtools}

\declaretheorem[name=Definition,numberwithin=section]{definition}

\declaretheorem[name=Proposition,sibling=definition]{proposition}

\declaretheorem[name=Theorem,sibling=definition]{theorem}

\newcommand{\Nat}{\mathbb{N}}
\newcommand{\dd}{\mathrm{d}}
\newcommand{\Ind}{\mathbb{I}}

\newcommand{\Prob}{\mathbb{P}}

\newcommand{\eqD}{\overset{\rm d}{\sim}}

\newcommand{\Bin}{\text{\rm Bin}}
\newcommand{\Nor}{\text{\rm N}}

\newcommand{\ox}{\otimes}
\newcommand{\bigox}{\bigotimes}